\documentclass[11pt,leqno]{amsart}

\usepackage{amsmath,amsfonts,latexsym,graphicx,amssymb,url}
\usepackage{stmaryrd}
\usepackage{mathrsfs}

\usepackage{hyperref}

\usepackage{txfonts}
\usepackage{esint}

\usepackage{color}
\usepackage{todonotes}

\renewcommand{\div}{\mbox{div}\,}

\newcommand{\R}{{\mathbb R}} %%reals

\newcommand{\A}{\mathbb{A}}

\newcommand\norm[1]{\left\| #1\right\|}

\newcommand{\I}{{\mathbb I}}

\newcommand{\wei}[1]{\langle #1 \rangle}

\newcommand{\F}{{\mathbf F}}
\newcommand{\G}{{\mathbf G}}

\newtheorem{theorem}{Theorem}[section]
\newtheorem{definition}[theorem]{Definition}
\newtheorem{remark}[theorem]{Remark}
\newtheorem{lemma}[theorem]{Lemma}
\newtheorem{proposition}[theorem]{Proposition}

\numberwithin{equation}{section}

\newcommand{\beq}{\begin{equation}}
\newcommand{\eeq}{\end{equation}}

\renewcommand{\epsilon}{\varepsilon}

\title[Weighted gradient estimates, degenerate Elliptic Equations]{Sobolev estimates for singular-degenerate quasilinear equations beyond the $A_2$ class}

\author[H. Dong]{Hongjie Dong}
\address{Division of Applied Mathematics, Brown University, 182 George Street, Providence, RI 02912, USA}
\email{hongjie\_dong@brown.edu}

\author[T. Phan]{Tuoc Phan}
\address{Department of Mathematics, University of Tennessee, Knoxville, 227 Ayres Hall, 1403 Circle Drive, Knoxville, TN 37996, USA}
\email{tphan2@utk.edu}

\author[Y. Sire]{Yannick Sire}
\address{ Department of Mathematics, Johns Hopkins University, 404 Krieger Hall, 3400 N. Charles Street, Baltimore, MD 21218, USA}
\email{ysire1@jhu.edu}

\thanks{H. Dong is partially supported by Simons Fellows Award 007638 and the NSF under agreement DMS-2055244.
Y. Sire is partially supported by the NSF under agreement DMS-2154219.}


\begin{document}
\begin{abstract}
We study a conormal boundary value problem for a class of quasilinear elliptic equations in bounded domain $\Omega$ whose coefficients can be  degenerate or singular of the type $\text{dist}(x, \partial \Omega)^\alpha$, where $\partial \Omega$ is the boundary of $\Omega$ and $\alpha \in (-1, \infty)$ is a given number.  We establish weighted Sobolev type estimates for weak solutions  under a smallness assumption on the weighted mean oscillations of the coefficients in small balls. Our approach relies on a perturbative method and several new Lipschitz estimates for weak solutions to a class of singular-degenerate quasilinear equations. \end{abstract}

\maketitle

%\tableofcontents
\section{Introduction and problem setting} \label{Intro-sec}
%\subsection{Themes and discussions}
Let  $\Omega$ be a nonempty open bounded set in $\mathbb{R}^n$ with Lipschitz boundary $\partial \Omega$.  We consider the following class of quasilinear equations with singular-degenerate coefficients and with conormal boundary condition
\begin{equation}  \label{main-eqn}
\left\{
\begin{array}{cccl}
\text{div} [\mu(x) \A(x, \nabla u(x))] & = & \text{div}[\mu(x) \F(x)] & \quad \text{in} \quad \Omega, \\
\mu(x)\A(x, \nabla u(x)) \cdot \vec{\nu}(x) & = & \mu(x) \F(x)\cdot \vec{\nu}(x) & \quad \text{on} \quad \partial \Omega.
\end{array} \right.
\end{equation}
Here, $\vec{\nu} : \partial \Omega \rightarrow \mathbb{R}^n$ is the unit outward vector, $\F:\Omega\rightarrow \mathbb{R}^{n}$ is a given measurable vector field, $\A : \Omega \times \mathbb{R}^n \rightarrow \mathbb{R}^n$ is a given vector field that is measurable in $x \in \Omega$ and Lipschitz in $\xi \in \R^n\setminus\{0\}$, and $\mu: \overline{\Omega} \rightarrow [0, \infty)$ is a  weight function. We assume that $\mu$ is continuous in $\Omega$  and there is a sufficiently small constant $r_0 \in (0, \text{diam}(\Omega))$ such that
\begin{equation} \label{mu.def}
\mu(x) =
\left\{
\begin{array}{cl}
 \text{dist}(x, \partial \Omega)^\alpha & \quad \text{when} \quad \text{dist}(x, \partial \Omega) < r_0\\
 1 & \quad \text{when} \quad \text{dist}(x, \partial \Omega) > 2r_0
 \end{array}
\right.
\end{equation}
with $\alpha \in (-1, \infty)$. We also assume that the vector field $\A : \Omega \times \mathbb{R}^n \rightarrow \mathbb{R}^n$ satisfies the following ellipticity and growth conditions: there exists $\kappa \in (0,1)$ such that
\begin{equation} \label{structure-A}
\left\{
\begin{aligned}
&  \kappa |\xi - \eta|^2 \leq   \wei{\A(x, \xi) -  \A(x, \eta), \xi - \eta}, \quad  \forall \ \xi, \eta \in \mathbb{R}^n, \ \forall \ x \in \Omega, \\
& \A(x, 0) =0, \quad   \quad |\A_{\xi}(x, \xi)| \leq \kappa^{-1}, \quad \
 \quad \forall \ \xi \in \mathbb{R}^n\setminus \{0\}, \quad \forall \ \ x \in \Omega.
\end{aligned} \right.
\end{equation}
Observe that under the assumptions \eqref{mu.def} and \eqref{structure-A}, the  equation \eqref{main-eqn} is singular on $\partial \Omega$ when $\alpha <0$ and degenerate when $\alpha >0$. If $\alpha=0$, \eqref{main-eqn}  reduces to the classical  uniformly elliptic quasilinear equation.

\smallskip
The main purpose of the present paper is to develop a weighted Sobolev theory for weak solutions to the class of singular-degenerate equations \eqref{main-eqn}, in which the weight $\mu$ may not be in the $A_2$ class of Muckenhoupt weights as classically considered in the literature.

\smallskip

Assuming ${\bf F} \in L_{p}(\Omega,\mu)$ for $p \in {[2,} \infty)$ and $\alpha \in (-1, \infty)$, we prove the following weighted estimate of Calder\'{o}n-Zygmund type
\begin{equation} \label{show-off-lp}
\left(\int_{\Omega} |\nabla u(x)|^{p} \mu(x) \,dx \right)^{1/p} \leq N \left( \int_{\Omega}|{\bf F}(x)|^{p} \mu(x) \,dx \right)^{1/p}
\end{equation}
for any weak solution $u$ to \eqref{main-eqn}, where $N>0$ is a positive constant  independent of $u$ and $\bf F$.  See Theorem \ref{g-theorem} below for the precise statement of the result.

\smallskip
To motivate the present investigation, let us discuss briefly a few applications of the study of \eqref{main-eqn}. Quasilinear problems of the form \eqref{main-eqn} appear naturally in the theory of relativistic Euler equations with a physical vacuum condition, as investigated recently in \cite{DIT}. Our system \eqref{main-eqn} corresponds to a stationary version of it. In fact, in \cite{DIT}, the authors considered a similar situation as ours with $\alpha=1$. We believe that our results complement theirs in a very natural way. Also, linear and nonlinear operators of the type considered in \eqref{main-eqn} appear in the study of some special geometric structures known as conic-edge metrics (see e.g. \cite{mazzeo}). The techniques involved in \cite{mazzeo} and many subsequent works in this area are of geometric microlocal nature and completely different from our techniques here. Harmonic maps between conic manifolds  were considered for instance in \cite{jessie} (see also the references therein). The equation under consideration here is an instance of those maps between a manifold with conic-edge metric and a smooth closed manifold. More interestingly, the presence of conormal data is reminiscent of a free boundary version of those as in \cite{moserRoberts,roberts}. Applications to geometric variational problems of this type will be addressed in a subsequent work.

\smallskip
We emphasize that the estimate \eqref{show-off-lp} is new even when \eqref{main-eqn} is linear with $\A(x, \xi) = \xi$ for $(x,\xi) \in \Omega \times \mathbb{R}^n$ as it deals with equations in general domains. Indeed, when the domains are upper-half spaces, more general results on the existence and regularity estimates in weighted and mixed-norm Sobolev spaces for a similar class of linear parabolic equations can be found first in \cite{DP19} with $\alpha \in (-1, 1)$ and then in \cite{DP20} with $\alpha \in (-1, \infty)$.  Similar results for problems with homogeneous Dirichlet boundary conditions can be found in \cite{DP-2023, DP-2021}.  See also a series of papers \cite{vita,vita2,tortone} in which the authors investigated some properties of degenerate-singular linear equations of the form \eqref{main-eqn} in  domains  with flat boundaries under  sufficiently smooth and symmetry assumptions on the leading coefficients. Schauder estimates, Liouville theorems, and geometric properties of the solutions are derived in  these papers.

\smallskip
We also note that when  $\alpha \in (-1,1)$, in the influential paper \cite{CS}  the authors showed that in the linear setting \eqref{main-eqn} is related to the realization as a Dirichlet-to-Neumann map of the fractional Laplacian. In this case, $\mu$ belongs to Muckenhoupt class of $A_2$ weights, and in the framework of non-local fractional elliptic equations, the weighted Sobolev theory was developed in \cite{MP}.  See also  \cite{CMP} for similar results on $W^1_p$-estimates for solutions of linear elliptic equations whose coefficients can be singular or degenerate with general $A_2$-weights instead of the distance function $\mu$ as in \eqref{main-eqn}, but with some restrictive smallness assumption on the weighted mean oscillations of the coefficients that cannot be applied to our setting here. The same class of linear elliptic equations whose coefficients are singular or degenerate as general $A_2$-Muckenhoupt weights were also studied in the classical papers \cite{Fabes, Surnachev, Murthy-Stamp} in which H\"{o}lder regularity of solutions were proved.

\section{Functional spaces, definitions, and statements of main results }\label{Sec-main}
 Let us introduce some notation and definitions used in the paper.  For a given nonnegative weight $\sigma$ on $\Omega$ and for $1 \leq p < \infty$, a measurable function $f$ defined on $\Omega$ is said to be in the weighted Lebesgue space $L_p(\Omega, \sigma)$  if
\[\norm{f}_{L_p(\Omega,\sigma)} =\left( \int_{\Omega} |f(x)|^p \sigma(x) \,dx  \right)^{1/p}< \infty. \]
For $k \in \mathbb{N}$, a function $f \in L_p(\Omega, \sigma)$ is said to belong to the weighted Sobolev space $W^k_p(\Omega, \sigma)$ if all of its distributional derivatives
$D^\beta f$ are in $L_p(\Omega, \sigma)$ for $\beta= (\beta_1, \beta_2, \ldots, \beta_n) \in  (\mathbb{N} \cup \{0\})^n$ and $|\beta| = \beta_1 + \beta_2 +\ldots + \beta_n \leq k$.  The space $W^k_p(\Omega, \sigma)$ is equipped with the norm
\[
\norm{f}_{W^k_p(\Omega, \sigma)} = \left( \sum_{|\beta| \leq k} \norm{D^\beta f}_{L_p(\Omega, \sigma)}^p \right)^{1/p}.
\]
%Moreover, we also denote $W^1_p_0(\Omega, \sigma)$\todo{Is this space used in the paper?} to be  the closure of $C_0^\infty(\Omega)$ in $W^1_p(\Omega, \sigma)$.

Next, we give the definition of weak solutions to \eqref{main-eqn}.
%===
\begin{definition} Assume that \eqref{structure-A} holds, $\F \in L_p(\Omega, \mu)^n$ with $ 1 < p < \infty$, and $\mu$ satisfies \eqref{mu.def}. A function $u~\in~W^1_p(\Omega, \mu)$ is said to be a weak solution of \eqref{main-eqn} if
\begin{equation}\label{def-soln}
\int_\Omega \mu(x)\wei{\mathbb{A}(x,\nabla u(x)), \nabla \varphi} \,dx =  \int_\Omega \mu(x)\langle {\bf F}(x), \nabla \varphi (x)\rangle \,dx, \quad \forall \varphi\in C^{\infty}(\overline{\Omega}).
\end{equation}
\end{definition} \noindent
\noindent
For each $\rho >0$ and $x \in \mathbb{R}^n$, we denote $B_\rho(x)$ to be the ball in $\mathbb{R}^n$ of radius $\rho$ and centered at $x$. When $x =0$, we simply write $B_\rho = B_\rho(0)$. Also, for each $x \in \overline{\Omega}$ and $\rho >0$, we write
\[
\Omega_\rho(x) = \Omega \cap B_\rho(x).
\]
We give the following definition on bounded mean oscillations with weight $\mu$ for the vector field $\A$.
%====
\begin{definition} \label{BMO-def} For every $x_0 \in \overline{\Omega}$ and $\rho >0$, and for a given measurable vector field $\A :\Omega \times \mathbb{R}^n \rightarrow \mathbb{R}^n$ satisfying \eqref{structure-A}, the mean oscillation of $\A$ in $\Omega_\rho(x_0)$ with respect to the weight $\mu$ is defined by
\begin{equation} \label{Theta.def}
\Theta_{\rho, x_0}(\A, \mu)= \frac{1}{\mu(\Omega_\rho(x_0))}\int_{\Omega_\rho(x_0)} \mu(x) \left(\sup_{\xi \in \mathbb{R}^n \setminus \{0\}} \frac{|\mathbb{A}(x, \xi) -\A_{\Omega_\rho(x_0)}(\xi)|}{|\xi|}\right) \,dx,
\end{equation}
where $\mu(\Omega_\rho(x_0)) = \displaystyle{\int_{\Omega_\rho(x_0)} \mu(x)\, dx}$, and $\A_{\Omega_\rho(x_0)}(\xi) $ is the weighted average of $\A$ in $\Omega_\rho(x_0)$ which is defined by
\begin{equation} \label{A-weight-everage}
\A_{\Omega_\rho(x_0)}(\xi) = \frac{1}{{\mu(\Omega_\rho(x_0))}}\int_{\Omega_\rho(x_0)} \mathbb{A}(x,\xi){\mu(x)} \,dx.
\end{equation}
\end{definition}
%We next recall the following definition on the Lipschitz of the boundary $\partial \Omega$.
Below for each $x' \in \mathbb{R}^{n-1}$ and $\rho>0$, we denote $B'_\rho(x')$ the ball in $\mathbb{R}^{n-1}$ with radius $\rho$ and centered at $x'$.
%====
\begin{definition} \label{Lip-domain} For given positive numbers $\delta$ and $\rho_0$, we say that $\Omega$ is of $(\delta, \rho_0)$-Lipschitz if for any $x_0 = (x'_0, x_n^0) \in \partial \Omega$, there exists a Lipschitz continuous function $\gamma : \mathbb{R}^{n-1} \rightarrow \mathbb{R}$ such that upon relabelling and reorienting the coordinates
\[
\begin{split}
&\Big \{x= (x', x_n) \in B_{\rho_0}'(x_0') \times \mathbb{R}:  \gamma(x') < x_n < \gamma(x') + \rho_0 \Big\} \subset\Omega,   \\
& \Big\{(x', \gamma(x')): x' \in B'_{\rho_0}(x_0') \Big\} \subset \partial \Omega,
\end{split}
\]
and
\[
\gamma(x'_0) = x_n^0,  \quad \nabla \gamma(x'_0) =0, \quad  \|\nabla \gamma\|_{L_\infty(\mathbb{R}^{n-1})} \leq \delta.
\]
\end{definition}
\begin{remark} If $\Omega$ is $(\delta, \rho_0)$-Lipschitz, then it is $(\delta, \rho)$-Lipschitz for any $\rho \in (0,\rho_0)$. If $\partial \Omega \in C^1$, then it is $(\delta, \rho_0)$-Lipschitz for any sufficiently small $\delta>0$ and for $\rho_0 = \rho_0(\Omega, \delta,n)>0$.
\end{remark}
\noindent
The following theorem on gradient estimates of weak solutions to \eqref{main-eqn} is the main result of the paper.
\begin{theorem} \label{g-theorem} Let $\alpha \in (-1, \infty)$,  $\kappa, r_0, \rho_0 \in (0,1)$, and $p \in [2, \infty)$. There exists a sufficiently small number $\delta =\delta(\kappa, n, p, r_0, \rho_0, \alpha)>0$ such that the following assertions hold.  Assume that \eqref{mu.def} and \eqref{structure-A} hold, and $\Omega$ is $(\delta, \rho_0)$-Lipschitz. Assume also that there is $R_0 \in (0,1)$ so that
\begin{equation} \label{Lip-BMO}
\sup_{\rho \in (0, R_0)}\sup_{x\in \overline{\Omega}}\Theta_{\rho, x} (\A, \mu) \leq \delta.
\end{equation}
Then for any weak solution $u \in W^1_2(\Omega, \mu)$ to \eqref{main-eqn}  with some ${\bf F} \in L_p(\Omega, \mu)^n$, we have $u\  \in\ W^{1}_{p}(\Omega, \mu)$, and it satisfies the estimates
\begin{equation} \label{main-est}
\begin{split}
& \norm{\nabla u}_{L_p(\Omega, \mu)} \leq N
\norm{{\bf F}}_{L_p(\Omega, \mu)}  \quad \text{and} \\
& \norm{u}_{L_p(\Omega, \mu)} \leq N\norm{u}_{L_2(\Omega, \mu)} + N\norm{{\bf F}}_{L_p(\Omega, \mu)},
\end{split}
\end{equation}
where $N$ is a positive constant depending on $n, p, \kappa, \alpha, r_0, R_0, \rho_0, \alpha, \Omega$, and $\mu$.
Moreover, for any ${\bf F} \in L_p(\Omega, \mu)^n$, there is a weak solution $u \in W^{1}_{p}(\Omega, \mu)$ to the equation \eqref{main-eqn} and it is unique up to a constant.
 \end{theorem}

In this paper, we also establish local interior and boundary regularity estimates for gradients of weak solutions to \eqref{main-eqn},  which could be useful for other purposes. In particular,  in Theorem \ref{flat-Lp-thrm} below, we prove local regularity estimates of the gradients of weak solutions to \eqref{main-eqn} when the boundaries of the domains are flat, and similarly in Theorem \ref{local-Lp-thm} for general domains. We also note that in Theorem \ref{flat-Lp-thrm}, we require the vector field $\A$  to be only partially VMO, a condition which was introduced in \cite{Kim-Krylov} in the study of linear equations with uniformly elliptic and bounded coefficients. The interior regularity gradient estimates are also proved in Theorem \ref{inter.est-thm}.

The proof of Theorem \ref{g-theorem} is based on a perturbation method using the freezing coefficient technique. To establish the integrability  of the gradients of solutions, we  employ the method introduced in \cite{CP} that relies on estimates of the level sets of $\nabla u$. See also \cite{CFL, Fazio-1} for the classical approach using solution representation formulas and \cite{Krylov} for an approach using the Fefferman-Stein sharp function  theorem. To  obtain the regularity estimate \eqref{main-est}, it is crucial to derive Lipschitz estimates of solutions to a class of homogeneous quasilinear equations with singular-degenerate coefficients as in \eqref{main-eqn}.  We accomplish this by first  employing the Moser iteration  argument to derive the estimates of the tangential derivatives of solutions. To estimate the normal derivative of solutions, we  exploit the structure of the PDE in \eqref{main-eqn} and its boundary condition. The results and techniques developed here to derive the Lipschitz estimates are of independent interest and they could be useful for other studies.

The remaining part of the paper is organized as follows. In Section \ref{Lip-est}, we derive the Lipschitz estimates of solutions to a class of homogeneous equations in  domains with flat boundaries. In Section \ref{local-Lp-sec}, we prove local $L_p$ regularity estimate for the gradients of solutions in  domains with flat boundaries. In the last section, Section \ref{boundary-global-section}, we provide the proof of Theorem \ref{g-theorem} which is based on the flattening of the domain boundaries and the local interior and boundary estimates developed in the previous sections.
%===================================

%\section{Preliminaries on weights and weighted norm inequalities} \label{preliminaries}

%===============
%
\section{Lipschitz estimates of homogeneous equations in flat domains} \label{Lip-est}

We denote $\R^n_+= \R^{n-1} \times (0,\infty)$ and for $r>0$ and $x_0 = (x', x_n^0) \in \overline{\R^n_+}$, we write
\[
B_r^+(x_0) = B_r(x_0) \cap \R^n_+, \quad T_r(x_0) =   B_r(x_0) \cap \{x_n =0\}.
\]
We study the following class of equations
\begin{equation}\label{constantCoeff}
\text{div}(x_n^\alpha \A_0(x_n, \nabla u(x)))=0\quad \text{for}\quad x= (x', x_n) \in B_r^+(x_0)
\end{equation}
with the homogeneous conormal boundary condition
\begin{equation}  \label{constantCoeffBC}
      x_n^\alpha a_{n}(x_n,\nabla u(x))=0 \quad \text{on}\quad T_r(x_0)  \quad \text{if} \quad T_r(x_0)  \not= \emptyset.
\end{equation}
Here, $\alpha>-1$ is a given constant, and
$$\A_0 = (a_1, a_2, \ldots, a_n) : ((x_n^0-r)_+, x_n^0+r) \times \R^n \rightarrow \R^n$$
is a given vector field, where we denote  $s_+ =\max\{s, 0\}$ for any real number $s$. The main result of this section, Lemma \ref{lemmaLipschitz}, gives Lipschitz estimates for weak solutions to \eqref{constantCoeff}-\eqref{constantCoeffBC}.

We assume that $\A_0$ is measurable in $x_n \in ((x_n^0-r)_+, x_n^0+r)$, Lipschitz in $\xi \in \R^n\setminus\{0\}$, and it satisfies the following ellipticity and growth conditions: there is $\kappa\in (0,1)$ such that
\begin{equation}\label{elip-interior}
\left \{
\begin{aligned}
&  \kappa |\xi - \eta|^2 \leq   \wei{\A_0(x_n, \xi) -  \A_0(x_n, \eta), \xi - \eta}, \quad  \forall \ \xi, \eta \in \mathbb{R}^n, \\
& \A_0(x_n, 0) =0, \quad   \quad |D_\xi\A_{0}(x_n, \xi)| \leq \kappa^{-1}, \quad \
 \quad \forall \ \xi \in \mathbb{R}^n\setminus \{0\}
 \end{aligned} \right.
\end{equation}
for all $x_n \in (x_n^0-r)_+, x_n^0+r)$. A function $u\in W^1_2(B^+_r(x_0), \mu)$ is said to be a weak solution to \eqref{constantCoeff}-\eqref{constantCoeffBC}  if
\begin{equation} \label{loc-weak-def}
\int_{B_r^+(x_0)} x_n^\alpha \wei{\A_0(x_n, \nabla u), \nabla \varphi} dx =0, \quad
\forall \ \varphi \in C_0^\infty(B_r(x_0)).
\end{equation}
For convenience, we also denote
$$d\mu(x) =x^\alpha_n \,dx' dx_n,$$
which is a doubling measure.  Also, for a measurable function $f$ with some suitable integrability condition defined in a non-empty open set $B \subset \mathbb{R}^n$, we write
\[
(f)_B = \frac{1}{|B|} \int_B f(x)\, dx \quad \text{and}\quad \fint_B f(x)\, d\omega(x) = \frac{1}{\omega(B)} \int_B f(x)\, d\omega(x)
\]
where $\omega$ is some locally finite measure.

We start with the following lemma on local energy estimates of weak solutions to \eqref{constantCoeff}-\eqref{constantCoeffBC}.
\begin{lemma}[Caccioppoli inequality]
                \label{lem1}
Let $r>0$, $x_0\in \overline{\R^{n}_+}$, and $u\in W^1_2(B^+_r(x_0), \mu)$ be a weak solution to \eqref{constantCoeff}-\eqref{constantCoeffBC}. Then for any $c\in \R$ and $c'=(c_1, c_2,\ldots,c_{n-1})\in \R^{n-1}$,
\begin{equation} \label{assert-0427-1}
\int_{B^+_{r/2}(x_0)}|\nabla u(x)|^2\,d\mu(x)
\le Nr^{-2}\int_{B^+_{r}(x_0)}|u(x)- c|^2\,d\mu(x)
\end{equation}
and
\begin{equation} \label{assert-0427-2}
\int_{B^+_{r/2}(x_0)}|\nabla \nabla_{x'}u(x)|^2\,d\mu (x)
\le Nr^{-2}\int_{B^+_{r}(x_0)}|\nabla_{x'}u(x)- c'|^2\,d\mu (x),
\end{equation}
where $N = N(\kappa, n)>0$ and $\nabla_{x'} = \displaystyle{(\partial_{x_1}, \partial_{x_2},\ldots,  \partial_{x_{n-1}})}$.
\end{lemma}
\begin{proof}
Let $\zeta\in C_0^\infty(B_r(x_0))$ be a non-negative cut-off function satisfying
\[ \zeta=1 \quad \text{in} \quad B_{r/2}(x_0) \quad \text{and} \quad \|\nabla \zeta\|_{L_\infty} \leq \frac{N_0}{r},
\]
for some generic constant $N_0>0$.  Using  $(u-c)\zeta^2$ as the test function in \eqref{loc-weak-def}, we obtain
\[
\int_{B_r^+(x_0)} \wei{\A_0(x_n, \nabla u), \nabla u } \zeta^2 d\mu(x)
+ 2 \int_{B_r^+(x_0)} \wei{\A_0(x_n, \nabla u), \nabla \zeta} \zeta (u- c) d\mu (x) =0.
\]
From this, and \eqref{elip-interior}, we infer that
\[
\int_{B_r^+(x_0)} |\nabla u(x)|^2  \zeta(x)^2 d\mu(x)  \leq N\int_{B_r^+(x_0)} |\nabla u(x)| |\nabla \zeta(x)| |u(x)- c| |\zeta| d\mu (x)
\]
where $N = N(\kappa, n) >0$. From this, we follow the standard method using Young inequality to derive the estimate \eqref{assert-0427-1}.

Next, we prove \eqref{assert-0427-2}. By using a different quotient method if needed, we can formally differentiate the equation \eqref{constantCoeff} with respect to $x_k, k=1, 2, \ldots,n-1,$ to see that $u_k:=\partial_k u$ satisfies a linear elliptic equation
\begin{equation}
                    \label{eq2.56}
\text{div}(x_n^\alpha \partial_{\xi_j} \A_0(x_n, \nabla u)\partial_j u_k)=0 \quad \text{in} \quad B_r^+(x_0)
\end{equation}
with the same conormal boundary condition \eqref{constantCoeffBC}, where $\partial_k =  \partial_{x_k}$ and the Einstein summation convention is used. Now we test the equation with $(u_k-c_k)\zeta^2$ and use the ellipticity and boundedness condition in \eqref{elip-interior} as we just did to derive \eqref{assert-0427-2}. The proof of the lemma is completed.
\end{proof}

\begin{lemma}
                        \label{lem2}
Let $r>0$, $x_0\in \overline{\R^{n}_+}$, and $u\in W^1_2(B^+_r(x_0), \mu)$ be a weak solution to \eqref{constantCoeff}-\eqref{constantCoeffBC}. Then
$$
\|\nabla_{x'}u\|_{L_\infty(B_{r/2}^+(x_0))}
\le N\left(\fint_{B^+_{r}(x_0)}|\nabla_{x'}u(x)|^2\,d\mu(x)\right)^{1/2},
$$
where $N = N(\kappa, n, \alpha)>0$.
\end{lemma}
\begin{proof} By Lemma \ref{lem1}, we see that $\partial_k u$ is in $W^{1}_{2}(B^+_r(x_0))$ locally and it is a weak solution to the linear equation \eqref{eq2.56} with the boundary condition \eqref{constantCoeffBC}, for $k=1,2,\ldots, n-1$. Then, the assertion of the lemma follows by applying the Moser iteration argument  to the linear equation \eqref{eq2.56}. See \cite[Lemma 4.3]{DP20} for a similar result but for linear parabolic equations, and also \cite[Proposition 2.17]{vita} for a result for linear elliptic equations. We skip the details.
\end{proof}

The following lemma is our main result on Lipschitz estimates for weak solutions to \eqref{constantCoeff}-\eqref{constantCoeffBC}.
\begin{lemma}
                                    \label{lemmaLipschitz}
Let $r>0$, $x_0 = (x_0', x_n^0) \in \overline{\R^{n}_+}$, and $u\in W^1_2(B^+_r(x_0), \mu)$ be a weak solution to \eqref{constantCoeff}-\eqref{constantCoeffBC}. Then
\begin{equation}
                    \label{eq3.27}
\|\nabla u\|_{L_\infty(B^+_{r/2}(x_0))}
\le N\left(\fint_{B^+_{r}(x_0)}|\nabla u (x)|^2\,d\mu(x)\right)^{1/2},
\end{equation}
where $N = N(\kappa, n, \alpha)>0$.
\end{lemma}
\begin{proof}
By Lemma \ref{lem2}, it remains to prove the estimate of $\partial_{n} u$. Denote $$
U(x) =x_n^\alpha a_n(x_n, \nabla u(x)), \quad x = (x', x_n) \in B_r^+(x_0).
$$
Note that from \eqref{constantCoeff}, it follows that
\begin{align} \notag
\partial_{n} U(x', x_n) & = -\sum_{i=1}^{n-1}\partial_i [x_n^\alpha a_i(x_n, \nabla u(x)) ]\\  \label{eq4.24}
& =-x_n^\alpha \sum_{i=1}^{n-1}\sum_{j=1}^{n}\partial_{\xi_j}a_{i}(x_n, \nabla u(x))\partial_{ij}u(x).
\end{align}
On the other hand, by a direct calculation, we also have
\begin{equation}
                \label{eq4.25}
\nabla_{x'}U (x', x_n)=x_n^\alpha \sum_{j=1}^{n-1}\partial_{\xi_j}a_{n}(x_n, \nabla u(x))\partial_{j}\nabla_{x'}u(x).
\end{equation}
By a covering argument, we only need to discuss two cases: the interior case when $x_{n}^0 \geq 2r$ and the boundary case when $x_0\in \partial\R^n_+$.

\smallskip
\noindent
{\em Case 1: Interior case.}  By scaling, for simplicity we also assume that $x_{n}^0=1$. Then for $x = (x', x_n) \in B_{2/3}(x_0)$, $x_n\sim 1$ such that $d\mu\sim dx$, and the equation \eqref{eq2.56} is uniformly elliptic in $B_{2/3}(x_0)$.  It follows that $\nabla_{x'} u$ is H\"older continuous in $\overline{B}_{5/8}(x_0)$.
By the Poincar\'e inequality, \eqref{eq4.24}, \eqref{eq4.25}, the boundedness condition in \eqref{elip-interior}, and Lemma \ref{lem1}, for any $y_0\in B_{r/2}(x_0)$ and $s<r/4$,
\begin{align*}
                \label{eq3.31}
&\fint_{B_{s/2}(y_0)}|U(x)-(U)_{B_{s/2}(y_0)}|^2\,dx
\le Ns^2 \fint_{B_{s/2}(y_0)}|\nabla U(x)|^2\,dx\\
&\le Ns^2 \fint_{B_{s/2}(y_0)}|\nabla \nabla_{x'}u(x)|^2\,dx\\
&\le N\fint_{B_{s}(y_0)}|\nabla_{x'}u(x)-(\nabla_{x'}u)_{B_{s}(y_0)}|^2\,dx,
\end{align*}
which together with the H\"older estimate of $\nabla_{x'} u$ implies that $U$ is H\"older continuous by using Campanato's characterization of H\"older spaces. In particular,
\begin{equation}
                \label{eq3.31}
\|a_n(y, \nabla u)\|_{L_\infty(B_{r/2}(x_0))}
\le N\|U\|_{L_\infty(B_{r/2}(x_0))}
\le N\left(\fint_{B_{r}(x_0)}|\nabla u(x)|^2\,dx\right)^{1/2},
\end{equation}
which together with \eqref{elip-interior} implies
\[
\|\partial_n u\|_{L_\infty(B_{r/2}(x_0))} \leq N\left(\fint_{B_{r}(x_0)}|\nabla u(x)|^2\,dx\right)^{1/2}.
\]
Therefore, \eqref{eq3.27} is proved as $d\mu \sim dx$ in this case.

\smallskip
\noindent
{\em Case 2: Boundary case.}  Without loss of generality, we take $x_0=0$ and $r=1$. We claim that for any $y_0 = (y_0', y_n^0) \in \overline{B^+_{1/2}}$ and $s<1/20$,
\begin{align}
                                \label{eq5.00}
\fint_{B_{s}^+(y_0)}|a_n(x_n, \nabla u(x))|\,d\mu(x)
\le N\left(\fint_{B_{1}^+}|\nabla_{x'}u(x)|^2 \,d\mu(x)\right)^{1/2}.
\end{align}
This and the Lebesgue differentiation theorem (with the doubling measure $\mu$) give
$$
\|a_n(\cdot, \nabla u(\cdot))\|_{L_\infty(B^+_{1/2})}
\le N\left(\fint_{B^+_{1}}|\nabla u(x)|^2\,d\mu(x)\right)^{1/2}
$$
and consequently \eqref{eq3.27}.

Next we prove the claim. First we assume that $y_{0}^{n}=0$.
Using \eqref{eq4.24}, \eqref{eq4.25}, the boundedness condition in \eqref{elip-interior}, {and Lemma \ref{lem1},} we have for any $s<1/4$,
\begin{align*}
&\fint_{B_{s}^+(y_0)}|\nabla U(x, x_n)|^2 x_n^{-\alpha}\,dx\le N\fint_{B_{s}^+(y_0)}|\nabla \nabla_{x'}u(x', x_n)|^2 x_n^{\alpha}\,dx\\
&\le Ns^{-2}\fint_{B_{2s}^+(y_0)}|\nabla_{x'}u(x', x_n)|^2 x_n^{\alpha}\,dx.
\end{align*}
This together with H\"older's inequality gives
\begin{align*}
\fint_{B_{s}^+(y_0)}|\nabla U(x)|\,dx
&\le \left(\fint_{B_{s}^+(y_0)}|\nabla U(x', x_n)|^2 x_n^{-\alpha}\,dx\right)^{1/2}
\left(\fint_{B_{s}^+(y_0)}x_n^{\alpha}\,dx\right)^{1/2}\\
&\le Ns^{\alpha/2-1}\left(\fint_{B_{2s}^+(y_0)}|\nabla_{x'}u(x', x_n)|^2 x_n^{\alpha}\,dx\right)^{1/2}.
\end{align*}
By using the zero boundary condition for $U$ at $\{x_n =0\}${ and the boundary Poincar\'e inequality}, we then get
\begin{align*}
\fint_{B_{s}^+(y_0)}|U(x', x_n)|\,dx& \le Ns\fint_{B_{s}^+(y_0)}|\partial_nU(x', x_n)|\,dx\\
&\le Ns^{\alpha/2}\left(\fint_{B_{2s}^+(y_0)}|\nabla_{x'}u(x', x_n)|^2 x_n^{\alpha}\,dx\right)^{1/2},
\end{align*}
which implies that
\begin{align*}
\fint_{B_{s}^+(y_0)}|a_n(x_n, \nabla u(x))|\,d\mu
\le N\left(\fint_{B_{2s}^+(y_0)}|\nabla_{x'}u(x)|^2\,d\mu(x)\right)^{1/2}.
\end{align*}
By Lemma \ref{lem2}, we obtain \eqref{eq5.00} in this case.

For the general case, when $s\ge y_{n}^0/4$, we have $B_s(y_0)\subset B_{5s}(y_0',0)$, and \eqref{eq5.00} follows from the first case because $5s<1/4$.
When $s<y_{n}^0/4$, we use \eqref{eq3.31} to get
\begin{align*}
&\fint_{B_{s}(y_0)}|a_n(x_n, \nabla u(x))|\,d\mu
\le \|a_n(\cdot, \nabla u(\cdot))\|_{L_\infty(B_{s}(y_0))}\\
&\le \|a_n(\cdot , \nabla u(\cdot))\|_{L_\infty(B_{y_{n}^0/4}(y_0))}
\le N\left(\fint_{B_{y_{n}^0/2}^+(y_0)}|\nabla_{x'}u(x)|^2\,d\mu(x)\right)^{1/2}.
\end{align*}
Then the claim follows from the previous case with $s=y_{n}^0/2$.
\end{proof}
\begin{remark}
Since we can have a H\"older estimate instead of the $L_\infty$ estimate in Lemma \ref{lem2}, {it is possible to} bound the H\"older norm of $a_n(y, \nabla u)$ in the above lemma. However, we will not use this in the proofs below.
\end{remark}
%======
\section{Boundary \texorpdfstring{$L_p$}{} regularity estimates in flat domains} \label{local-Lp-sec}

In this section, we establish local boundary $L_p$ estimates for the gradients of solutions to the quasilinear equation of the form \eqref{main-eqn} when the boundary $\partial \Omega$ is flat. Its main result is Theorem \ref{flat-Lp-thrm} below. To state the result, we need some notation and definitions. For each $x= (x', x_n) \in \mathbb{R}^{n-1} \times \mathbb{R}$, we write $D_\rho(x) = B_\rho'(x') \times (x_n -\rho, x_n + \rho)$ and
\[
D_\rho^+ (x)= D_\rho(x) \cap \mathbb{R}^n_+,
\]
where $B_\rho'(x')$ denotes the ball in $\mathbb{R}^{n-1}$ of radius $\rho>0$ and centered at $x' \in \mathbb{R}^{n-1}$. We also denote
$$
\Hat \Gamma_\rho(x) = \partial D_\rho^+(x) \cap \{x_n =0\}.
$$
When $x = 0$, we simply write $D_\rho^+ = D_\rho^{+}(0)$,  etc.
Let
$$\A = (A_1, A_2,\ldots, A_n) : D_2^+ \times \mathbb{R}^n \rightarrow \mathbb{R}^n$$
be measurable with respect to $x \in D_2^+$ and Lipschitz in $\xi \in \mathbb{R}^n \setminus \{0\}$. We assume that $\A$ satisfies the ellipticity and growth conditions as \eqref{structure-A} in $D_2^+$. Precisely, there exists $\kappa \in (0,1)$ such that
\begin{equation} \label{A-local}
\left\{
\begin{aligned}
&  \kappa |\xi - \eta|^2 \leq   \wei{\A(x, \xi) -  \A(x, \eta), \xi - \eta}, \quad  \forall \ \xi, \eta \in \mathbb{R}^n, \ \forall \ x \in D_2^+, \\
& \A(x, 0) =0, \quad   \quad |\A_{\xi}(x, \xi)| \leq \kappa^{-1}, \quad \
 \quad \forall \ \xi \in \mathbb{R}^n\setminus \{0\}, \quad \forall \ \ x \in D_2^+.
\end{aligned} \right.
\end{equation}
Let $\omega: \overline{D_2^+} \rightarrow \mathbb{R}_+$ be a weight defined by
\begin{equation} \label{omega.cond}
\omega(x) = x_n^\alpha (1- h(x))^\alpha,  \quad  \forall x = (x', x_n) \in D_2^+
\end{equation}
with $\alpha \in (-1, \infty)$ and $h: \overline{D^+_2} \rightarrow [0, 1)$ is a measurable function satisfying $\|h\|_{L_\infty(D_2^+)} <1$.

We study the following equation
\begin{equation} \label{D2.eqn}
\left\{
\begin{array}{cccl}
\div[\omega(x)\A(x, \nabla u(x))] & = & \div[\omega(x)\F] & \quad \text{in} \quad D_2^+ \\
\omega(x) A_n(x, \nabla u(x)) & = & \omega(x)F_n (x)& \quad \text{on} \quad  \Hat \Gamma_2,
\end{array}
\right.
\end{equation}
{where $\F = (F_1, F_2, \ldots, F_n) : D_2^+ \rightarrow \mathbb{R}^n$ is a given measurable vector-field}.
For $p \in (1, \infty)$, we say that $u \in W^{1}_{p}(D_2^+, \omega)$ is a weak solution of \eqref{D2.eqn} if
\[
\int_{D_2^+} \omega(x) \wei{\A(x,\nabla u(x)), \nabla \varphi(x)} \,dx = \int_{D_2^+}\omega(x) \wei{\F(x), \nabla \varphi(x)} \,dx
\]
for any $\varphi \in C^\infty_0( D_2)$.

We also need the following partial mean oscillation of the coefficients $\A$ which was introduced in \cite{Kim-Krylov}. %Note that the partial mean oscillation is {\color{blue}slightly different} {compared to the mean oscillation in Definition \ref{BMO-def}.
\begin{definition} \label{PBMO-def} For $\rho>0$ and $x_0 = (x'_0, x_n^0) \in \overline{D_1^+}$, the partial mean oscillation of a given vector field $\A: D_2^+ \times  \mathbb{R}^n \rightarrow \mathbb{R}^n$ in $D_{\rho}^+(x_0)$ with respect to the weight $\omega$ is defined by
\begin{equation*}
\begin{split}
\hat{\Theta}_{\rho, x_0}(\A) =&  \frac{1}{\omega(D_\rho^+(x_0))}\int_{D_\rho^+(x_0)} \omega(x) \left(\sup_{\xi \in \mathbb{R}^n \setminus \{0\}} \frac{|\A(x, \xi) - \bar{\A}_{B_\rho'(x_0')}(x_n, \xi)|}{|\xi|}\right) \,dx,
\end{split}
\end{equation*}
where $\bar{\A}_{B_\rho'(x_0')}(x_n, \xi)$ is the average of $\A$ in the $(n-1)$-dimensional ball $B_\rho'(x_0')$ defined by
\[
\bar{\A}_{B_\rho'(x_0')}(x_n, \xi) = \frac{1}{|B_\rho'(x_0')|}\int_{B_\rho'(x_0')} \A(x,\xi) \,dx'.
\]
\end{definition}

The main result of this section is the following theorem.

\begin{theorem} \label{flat-Lp-thrm}
For every $\kappa  \in (0, 1), p \in [2, \infty)$, and $\alpha \in (-1, \infty)$, there exists a sufficiently small positive number $\delta_0 = \delta_0(\kappa, p, \alpha, n)$ such that the following assertion holds.  Assume that $\A$ and $\omega$ satisfy \eqref{A-local}, \eqref{omega.cond}, $\|h\|_{L_\infty(D_2^+)} \leq \delta_0$, and
\[
\sup_{\rho \in (0,R_0)} \sup_{x \in \overline{D_1^+}}\hat{\Theta}_{\rho, x}(\A) \leq \delta_0,
\]
for some $R_0 \in (0,1)$. Then, if $u \in W^1_2(D_2^+, \omega)$ is a weak solution of \eqref{D2.eqn} with $\F \in L_p(D_2^+, \omega)$, we have $\nabla u \in L_p(D_1^+, \omega)$ and
\[
\|\nabla u\|_{L_p(D_1^+, \omega)} \leq N \|\nabla u\|_{L_2(D_2^+,\omega)} + N\|\F\|_{L_p(D_2^+, \omega)}
\]
for some constant $N = N(\kappa, p, \alpha, R_0, n)>0$.
\end{theorem}
The remaining part of the section is  devoted to the proof of Theorem \ref{flat-Lp-thrm}. We prove Theorem \ref{flat-Lp-thrm} using the level set argument introduced in \cite{CP}. For our implementation, the following result is the main ingredient.

\begin{proposition} \label{level-propos} For every $\kappa \in (0,1)$ and $\alpha \in (-1, \infty)$, there exist sufficiently small numbers $\delta_0' = \delta_0'(\kappa, n, \alpha) >0$ and $\lambda_0 = \lambda_0(\kappa, n, \alpha)>0$ such that the following assertions hold. Suppose that $\A$ and $ \omega$ satisfy \eqref{A-local}, \eqref{omega.cond}, and $\|h\|_{L_\infty(D_2^+)} \leq \delta_0'$. Then, for any weak solution $u \in W^1_2(D_2^+, \omega)$ of \eqref{D2.eqn}, and for any $\rho \in (0,1/2), x_0 \in \overline{D_1^+}$, there exists $w \in W^1_2(D_\rho^+(x_0), \omega)$ satisfying
\begin{equation} \label{u-w-est}
\begin{split}
& \left(\fint_{D_{\rho}^+(x_0)} |\nabla u (x)- \nabla w(x)|^2 \,d\omega(x) \right)^{1/2}\\
 & \leq N \Big(\hat{\Theta}^{\frac{\lambda}{2(2+\lambda)}}_{\rho, x_0}(\A) + \|h\|_{L_\infty(D_2^+)}\Big)\left(\fint_{D_{2\rho}^+(x_0)} |\nabla u(x)|^2 \,d\omega(x) \right)^{1/2}  \\
 & \qquad + N \left(\fint_{D_{2\rho}^+(x_0)}  |\F(x)|^2 \,d\omega(x) \right)^{1/2}
\end{split}
\end{equation}
and
\begin{equation} \label{w-Linfty}
\begin{split}
 \|\nabla w\|_{L_\infty(D^{+}_{\rho/2}(x_0))} & \leq N \left(\fint_{D_{2\rho}^+(x_0)} |\nabla u(x)|^2 \,d\omega(x) \right)^{1/2}  \\
 & \qquad + N \left(\fint_{D_{2\rho}^+(x_0)} |\F(x)|^2 \,d\omega(x) \right)^{1/2}
\end{split}
\end{equation}
for any $\lambda \in (0, \lambda_0)$, and $N = N(\kappa, \alpha, n, \lambda)>0$.
\end{proposition}
Below, we give the proof of Proposition \ref{level-propos}, which is divided into two steps. In the first step, for each $x_0 = (x_0', x_n^0) \in \overline{D_1^+}$ and $\rho \in (0,1/2)$, we perturb the equation \eqref{D2.eqn} and compare the solution $u \in W^1_2(D_2^+, \omega)$ with the solution $v \in W^1_2(D_{2\rho}^+(x_0), \omega)$ of the following boundary value problem
\begin{equation} \label{v-eqn}
\left\{
\begin{array}{cccl}
\div[x_n^\alpha \A(x, \nabla v(x))] & = & 0 & \quad \text{in} \quad D_{2\rho}^+(x_0), \\
v &= & u & \quad \text{on} \quad \partial D_{2\rho}^+ (x_0) \setminus \Hat\Gamma_{2\rho}(x_0), \\
x_n^\alpha A_n(x, \nabla v(x)) & = & 0 & \quad \text{on}  \quad \Hat \Gamma_{2\rho}(x_0) \quad \text{if} \quad  \Hat\Gamma_{2\rho}(x_0) \not= \emptyset.
\end{array} \right.
\end{equation}
We  note that $\omega (x) \sim x_n^\alpha$ when $\|h\|_{L_\infty} < 1$ and therefore
\[
W^1_2(D_{2\rho}^+(x_0), x_n^\alpha) = W^1_2(D_{2\rho}^+(x_0), \omega).
\]
We say that $v \in W^1_2(D_{2\rho}^+(x_0), \omega)$ is a weak solution of \eqref{v-eqn} if $v- u =0$ on $ D_{2\rho}^+ \setminus \Hat\Gamma_{2\rho}(x_0)$ in the sense of trace and
\[
\int_{D_{2\rho}^+(x_0)} \wei{\A(x, \nabla v(x)), \nabla \varphi(x)} x_{n}^\alpha dx =0, \quad \forall \ \varphi \in C_0^\infty(D_{2\rho}(x_0)).
\]
%for all $\varphi \in C_0^\infty(D_{2\rho}(x_0))$.
\begin{lemma} \label{step-1} There exist sufficiently small positive numbers $\lambda_0 = \lambda_0(\kappa, n, \alpha)$ and  $\delta_0' = \delta_0'(\kappa, n, \alpha)$ such that the following assertions hold. For each $u \in W^1_2(D_{2}^+, \omega)$, $x_0 \in \overline{D_1^+}$ and $\rho \in (0,1/2)$, there exists a unique weak solution $v \in W^1_2(D_{2\rho}^+(x_0), \omega)$ to \eqref{v-eqn} satisfying
\begin{equation} \label{Reverse-Holder}
\left(\fint_{D_\rho^+(x_0)} |\nabla v(x)|^{2+\lambda}\, d\omega\right)^{1/(2+\lambda)} \leq N(\kappa, n, \alpha, \lambda)\left(\fint_{D_{2\rho}^+(x_0)} |\nabla v(x)|^{2}\, d\omega\right)^{1/2}
\end{equation}
for any $\lambda \in (0, \lambda_0)$. Moreover, if $\|h\|_{L_\infty(D_2^+)} \leq \delta_0'$, then
\begin{equation} \label{u-v}
\begin{split}
& \fint_{D_{2\rho}^+(x_0)} |\nabla u(x) - \nabla v(x)|^2\, d\omega \\
& \leq  N(\kappa, \alpha) \|h\|_{L_\infty(D_2^+)}^2 \fint_{D_{2\rho}^+(x_0)} |\nabla u(x)|^2\, d\omega + N(\kappa) \fint_{D_{2\rho}^+(x_0)} |\F(x)|^2\, d \omega.
\end{split}
\end{equation}
\end{lemma}
\begin{proof}  To prove the existence of solution  $v \in W^1_2(D_{\rho}^+(x_0), \omega)$ to the equation \eqref{v-eqn}, let us denote
\[
\hat{\A}(x, \xi) = (\hat{A}_1(x, \xi), \hat{A}_2(x, \xi),\ldots, \hat{A}_n(x, \xi) ) = \A(x, \nabla u(x) + \xi) - \A(x, \nabla u(x)),
\]
and
\[
\G =(G_1, G_2,\ldots, G_n) = - \A(x, \nabla u(x)), \quad x \in D_2^+, \xi \in \mathbb{R}^n.
\]
We consider the equation
\begin{equation} \label{tilde-v-eqn}
\left\{
\begin{array}{ccll}
\div[x_n^\alpha \hat{\A}(x, \nabla \tilde{v}(x))] & = &\div[x^\alpha_n\G] & \  \text{in} \  D_{2\rho}^+(x_0), \\
\tilde{v} &= & 0 & \  \text{on} \  \partial D_{2\rho}^+ (x_0) \setminus \Hat\Gamma_{2\rho}(x_0), \\
x_n^\alpha \hat{A}_n(x, \nabla \tilde{v}(x)) & = & x^\alpha_n G_n & \  \text{on}  \  \Hat \Gamma_{2\rho}(x_0) \  \text{if} \   \Hat\Gamma_{2\rho}(x_0) \not= \emptyset.
\end{array} \right.
\end{equation}
Note that due to \eqref{A-local} and $u \in W^1_2(D_2^+, \omega)$, we have $\G \in L_2(D_2^+, \omega)$. Moreover, it is simple to check that $\hat{\A}$ satisfies the ellipticity and growth conditions as in \eqref{A-local}.
Also, let
\[
\mathbb{E} = \big\{ g \in W^1_2(D_{\rho}^+(x_0), \omega): g \vert_{\partial D_{2\rho}^+ \setminus \Hat\Gamma_{2\rho}(x_0)} =0 \, \text{ in the sense of trace}\big\}.
\]
It can be checked that $\mathbb{E}$ is uniformly convex, and thus it  is reflexive. Therefore, it follows from the Minty-Browder theorem (see \cite{Br76}) that there is a unique weak solution $\tilde{v} \in \mathbb{E}$ to \eqref{tilde-v-eqn}. From this, we obtain the existence and uniqueness of a weak solution $v := \tilde{v} + u \in W^1_2(D_{\rho}^+(x_0), \omega)$ to the equation \eqref{v-eqn}.

\smallskip
Next, observe that the estimate \eqref{Reverse-Holder} is well known as the reverse H\"{o}lder's inequality. The proof of \eqref{Reverse-Holder} follows from the standard method using Caccioppoli inequalities,  the weighted Sobolev inequality (see, for instance, \cite[Lemma 3.1 and Remark 3.2 (ii)]{DP20}, the doubling property of $\omega$ as $\alpha \in (-1, \infty)$, and Gerhing's lemma. As those are standard techniques, we skip the details.

\smallskip
It remains to prove the estimate \eqref{u-v}. We observe that by testing the equations \eqref{D2.eqn} and \eqref{v-eqn} with $u-v$, we obtain
\[
\int_{D_{2\rho}^+(x_0)} \omega(x)\wei{\A(x, \nabla u), \nabla u - \nabla v} \,dx = \int_{D_{2\rho}^+(x_0)} \omega(x)\wei{\F, \nabla u - \nabla v} \,dx
\]
and
\[
\int_{D_{2\rho}^+(x_0)} \omega(x)\wei{\A(x, \nabla v), \nabla u - \nabla v} \,dx =  \int_{D_{2\rho}^+(x_0)} (\omega(x) - x_n^\alpha)\wei{\A(x, \nabla v), \nabla u - \nabla v} \,dx.
\]
Therefore, it follows from \eqref{A-local} that
\[
\begin{split}
& \kappa \int_{D_{2\rho}^+(x_0)} |\nabla u(x) - \nabla v(x)|^2 \omega(x) \,dx \\
&\leq \int_{D_{2\rho}^+(x_0)} \omega(x) \wei{\A(x, \nabla u) - \A(x, \nabla v), \nabla u - \nabla v} \,dx \\
& \leq \int_{D_{2\rho}^+(x_0)} \omega(x) |\F(x)| |\nabla u(x) - \nabla v(x)| \,dx \\
& \qquad + \kappa^{-1}\sup_{x\in D_2^+}|1 - x_n^\alpha/\omega(x)|\int_{D_{2\rho}^+(x_0)} \omega(x)  |\nabla v(x)||\nabla u(x) - \nabla v(x)| \,dx \\
& \leq \frac{\kappa}{2} \int_{D_{2\rho}^+(x_0)} |\nabla u(x) - \nabla v(x)|^2 \omega(x) \,dx + N(\kappa) \int_{D_{2\rho}^+(x_0)}  |\F(x)|^2 \omega(x) \,dx \\
& \qquad + N(\kappa) \sup_{x\in D_2^+}|1 - x_n^\alpha/\omega(x)|^2 \int_{D_{2\rho}^+(x_0)} |\nabla v(x)|^2 \omega(x) \,dx.
\end{split}
\]
This implies that
\[
\begin{split}
\int_{D_{2\rho}^+(x_0)} |\nabla u(x) - \nabla v(x)|^2 \omega(x) \,dx & \leq  N(\kappa) \int_{D_{2\rho}^+(x_0)}  |\F(x)|^2 \omega(x) \,dx \\
& \quad + N(\kappa, \alpha)\|h\|_{L_\infty(D_2^+)}^2 \int_{D_{2\rho}^+(x_0)} |\nabla v(x)|^2 \omega(x) \,dx.
\end{split}
\]
Then, for $\delta_0' \in (0,1)$ sufficiently small such that
$$
N(\kappa, \alpha)\|h\|_{L_\infty(D_2^+)}^2\le N(\kappa, \alpha)(\delta_0')^2 \leq 1/4,
$$
we have
\[
\begin{split}
& \int_{D_{2\rho}^+(x_0)} |\nabla u (x)- \nabla v(x)|^2 \omega(x) \,dx  \leq  N(\kappa) \int_{D_{2\rho}^+(x_0)}  |\F(x)|^2 \omega(x) \,dx \\
 & \quad + \frac{1}{2}\int_{D_{2\rho}^+(x_0)} |\nabla u(x) - \nabla v(x)|^2 \omega(x)\,dx + N(\kappa, \alpha) \|h\|_{L_\infty(D_2^+)}^2 \int_{D_{2\rho}^+(x_0)} |\nabla u(x)|^2 \omega(x)\,dx.
\end{split}
\]
From this, the assertion \eqref{u-v} follows, and the proof of the lemma is completed.
\end{proof}
Next, recall that
\[
\bar{\A}_{B_\rho'(x_0')}(x_n, \xi) = \fint_{B_\rho'(x_0')} \A(x, \xi) \,dx'
\]
and we write
\[
 \bar{\A}_{B_\rho'(x_0')}(x_n, \xi)= (\bar{A}_{B_\rho'(x_0'),1}(x_n, \xi),  \bar{A}_{B_\rho'(x_0'),2}(x_n, \xi), \ldots,  \bar{A}_{B_\rho'(x_0'),n}(x_n, \xi)).
\]
In the next step of the perturbation, we consider the  equation with frozen coefficients
\begin{equation} \label{w-eqn}
\left\{
\begin{array}{cccl}
\div[x_n^\alpha \bar{\A}_{B_\rho'(x')}(x_n, \nabla w(x))] & = & 0 & \quad \text{in} \quad D_\rho^+(x_0), \\
w & = & v & \quad \text{on}\quad \partial D_{\rho}^+(x_0) \setminus \Hat\Gamma_{\rho}(x_0), \\
x_n^\alpha \bar{A}_{B_\rho'(x'),n}(x_n, \nabla w(x)) & = & 0 & \quad \text{on} \quad \Hat\Gamma_{\rho}(x_0) \quad \text{if}\quad \Hat\Gamma_{\rho}(x_0) \not=\emptyset,
\end{array}
\right.
\end{equation}
where $v$ is defined in Lemma \ref{step-1}.  The definition of a weak solution $w \in W^1_2( D_{\rho}^+(x_0), \omega)$ to \eqref{w-eqn} can be formulated exactly the same as that of \eqref{v-eqn}. In this step, we obtain the following approximation estimate.
\begin{lemma} \label{step-2}  Let $\delta_0'$ and $\lambda_0$ be as in Lemma \ref{step-1} and assume that $\|h\|_{L_\infty(D_2^+)} \leq \delta_0'$. Then, for each $u \in W^1_2(D_2^+, \omega)$ and $\rho \in (0,1), x_0 \in \overline{D_1^+}$, there exists a weak solution $w \in W^1_2(D_{\rho}^+(x_0), \omega)$ to   \eqref{w-eqn} satisfying
\[
\begin{split}
& \left(\fint_{D_{\rho}^+(x_0)}  |\nabla w(x) - \nabla v(x)|^2\, d \omega(x) \right)^{1/2} \\
& \leq N \hat{\Theta}^{\frac{\lambda}{2(2+\lambda)}}_{\rho, x_0}(\A) \left[ \left(\fint_{D_{2\rho}^+(x_0)}  |\nabla u(x)|^{2}\, d \omega(x) \right)^{\frac{1}{2}} + \left(\fint_{D_{2\rho}^+(x_0)}  |\F(x)|^{2}\, d \omega(x) \right)^{\frac{1}{2}} \right],
\end{split}
\]
where $N = N(\kappa, n, \alpha, \lambda)>0$ and $\lambda \in (0, \lambda_0)$.
\end{lemma}
\begin{proof} As in the proof of Lemma \ref{step-1}, the existence of a  weak solution $$w~\in~W^1_2(D_{\rho}^+(x_0), \omega)$$ to  \eqref{w-eqn} follows by the Minty-Browder theorem. As $w-v =0$ in the sense of trace on $\partial D_{\rho}^+(x_0) \setminus \Hat \Gamma_{\rho}(x_0)$, we can  use $w-v$ as a test function for \eqref{w-eqn} to obtain
\[
\int_{D_{\rho}^+(x_0)} x_n^\alpha \wei{\bar{\A}_{B_\rho'(x')}(x_n, \nabla w),  \nabla w-  \nabla v} \,dx =0.
\]
On the other hand,  as $\Hat\Gamma_{\rho}(x_0) \subset \Hat\Gamma_{2\rho}(x_0)$, we can also use $w-v$ as a test function for \eqref{v-eqn} and obtain
\[
\int_{D_{\rho}^+(x_0)} x_n^\alpha \wei{\A(x, \nabla v),  \nabla w-  \nabla v} \,dx =0.
\]
Then, it follows that
\[
\begin{split}
& \int_{D_{\rho}^+(x_0)} x_n^\alpha \wei{\A_{B_\rho'(x')}(x_n, \nabla w) - \bar{\A}_{B_\rho'(x')}(x_n, \nabla v),  \nabla w-  \nabla v} \,dx \\
& = \int_{D_{\rho}^+(x_0)} x_n^\alpha \wei{\A(x, \nabla v) -  \bar{\A}_{B_\rho'(x')}(x_n, \nabla v),  \nabla w-  \nabla v} \,dx.
\end{split}
\]
Next, by using the conditions in \eqref{A-local}, H\"{o}lder's inequality, and the fact that $\omega(x) \sim x_n^\alpha$, we obtain
\[
\begin{split}
& \fint_{D_{\rho}^+(x_0)}  |\nabla w(x) - \nabla v(x)|^2\, d \omega(x) \\
& \leq N\fint_{D_{\rho}^+(x_0)}  \frac{|\A(x, \nabla v) -  \bar{\A}_{B_\rho'(x')}(x_n, \nabla v)|}{|\nabla v(x)|}|\nabla v(x) |\nabla w(x) -\nabla v(x)|\, d \omega (x) \\
& \leq  N \left( \fint_{D_{\rho}^+(x_0)} \left[\sup_{\xi \in \mathbb{R}^n \setminus \{0\}} \frac{|\A(x, \xi) -  \bar{\A}_{B_\rho'(x')}(x_n, \xi)|}{|\xi|} \right]^{\frac{2(2+\lambda)}{\lambda}} \,d\omega(x) \right)^{\frac{\lambda}{2(2+\lambda)}} \\
& \qquad \times\left(\fint_{D_{\rho}^+(x_0)}  |\nabla v(x)|^{2+\lambda}\, d \omega(x) \right)^{\frac{1}{2+\lambda}} \left(\fint_{D_{\rho}^+(x_0)}  |\nabla w(x) -\nabla v(x)|^{2}\, d \omega(x) \right)^{\frac{1}{2}}  \\
& \leq  N \hat{\Theta}^{\frac{\lambda}{2(2+\lambda)}}_{\rho, x_0}(\A) \left(\fint_{D_{\rho}^+(x_0)}  |\nabla v(x)|^{2+\lambda}\, d \omega(x) \right)^{\frac{1}{2+\lambda}} \left(\fint_{D_{\rho}^+(x_0)}  |\nabla w(x) -\nabla v(x)|^{2}\, d \omega(x) \right)^{\frac{1}{2}},
\end{split}
\]
where in the last step we used the growth condition in \eqref{A-local} and $N = N(\kappa, n, \alpha, \lambda)$ is a positive constant for $\lambda \in (0, \lambda_0)$. It then follows that
\[
\left(\fint_{D_{\rho}^+(x_0)}  |\nabla w(x) - \nabla v(x)|^2\, d \omega(x) \right)^{1/2} \leq N \hat{\Theta}^{\frac{\lambda}{2(2+\lambda)}}_{\rho, x_0}(\A) \left(\fint_{D_{\rho}^+(x_0)}  |\nabla v(x)|^{2+\lambda}\, d \omega(x) \right)^{\frac{1}{2+\lambda}}.
\]
From this and \eqref{Reverse-Holder}, we infer that
\[
\left(\fint_{D_{\rho}^+(x_0)}  |\nabla w(x) - \nabla v(x)|^2\, d \omega(x) \right)^{1/2} \leq N \hat{\Theta}^{\frac{\lambda}{2(2+\lambda)}}_{\rho, x_0}(\A) \left(\fint_{D_{2\rho}^+(x_0)}  |\nabla v(x)|^{2}\, d \omega(x) \right)^{\frac{1}{2}}.
\]
Now, by using the triangle inequality and \eqref{u-v}, we obtain
\[
\begin{split}
& \left(\fint_{D_{\rho}^+(x_0)}  |\nabla w(x) - \nabla v(x)|^2\, d \omega(x) \right)^{1/2} \\
& \leq N \hat{\Theta}^{\frac{\lambda}{2(2+\lambda)}}_{\rho, x_0}(\A) \left[ \left(\fint_{D_{2\rho}^+(x_0)}  |\nabla u(x)|^{2}\, d \omega(x) \right)^{\frac{1}{2}} + \left(\fint_{D_{2\rho}^+(x_0)}  |\F(x)|^{2}\, d \omega(x) \right)^{\frac{1}{2}} \right].
\end{split}
\]
The lemma is proved.
\end{proof}
We are now ready to prove Proposition \ref{level-propos}.
\begin{proof}[Proof of Proposition \ref{level-propos}] Let $\delta_0'>0$ and $\lambda_0$ be as in Lemma \ref{step-1}. Then, we see that \eqref{u-w-est} follows easily from Lemma \ref{step-1}, Lemma \ref{step-2}, and the triangle inequality. {Similarly,} {by the triangle inequality}, we also obtain
\[
\begin{split}
& \left(\fint_{D_{\rho}^+(x_0)} |\nabla w(x)|^2 \,d\omega(x) \right)^{1/2} \\
& \leq N \left(\fint_{D_{2\rho}^+(x_0)} |\nabla u(x)|^2 \,d\omega(x) \right)^{1/2} + N \left(\fint_{D_{2\rho}^+(x_0)} |\F(x)|^2 \,d\omega(x) \right)^{1/2}
\end{split}
\]
for $N = N(\kappa, \alpha, n) >0$. As $\|h\|_{L_\infty(D_2^+)} < 1$, we see that $\omega \sim x_n^\alpha$. Therefore, by applying the results on the Lipschitz estimates (Lemma \ref{lemmaLipschitz}) for the homogeneous equation \eqref{w-eqn}, we obtain \eqref{w-Linfty}.
\end{proof}

\begin{proof}[Proof of Theorem \ref{flat-Lp-thrm}] From Proposition \ref{level-propos}, Theorem \ref{flat-Lp-thrm} follows from the  level set  argument introduced in \cite{CP}. {See also \cite{DK11}.} Since it is standard now, we skip the details.
\end{proof}

%===============
\section{Global \texorpdfstring{$L_p$}{Lp} regularity estimates} \label{boundary-global-section}

This section is devoted to the proof of Theorem \ref{g-theorem}. As usual, we divide the proof into two main steps: establishing interior estimates and boundary ones.
\subsection{Interior estimates} For every $\rho {\in (0, \text{diam}(\Omega)/2)}$, we write
\[
\Omega^{\rho} =\Big \{ x \in \Omega: \text{dist}(x, \partial \Omega) > \rho\Big\}.
\]
For $p \in (1, \infty), \rho >0$, we say that $u \in W^1_p(\Omega^\rho, \mu)$ is a weak solution of
\begin{equation} \label{inter-eqn}
\text{div}[\mu(x) \A(x, \nabla u(x))] = \text{div}[\mu(x) \F(x)] \quad \text{in} \quad \Omega^\rho
\end{equation}
if
\[
\int_{\Omega^\rho} \wei{\A(x,\nabla u(x)), \nabla \varphi(x)} \mu(x) \,dx = \int_{\Omega^\rho} \wei{\F(x), \nabla \varphi(x)} \mu(x)\,dx \quad \forall \ \varphi \  \in \ C_0^\infty(\Omega^\rho).
\]
The following theorem on interior estimates for \eqref{inter-eqn} is the main result of this subsection.

\begin{theorem} \label{inter.est-thm} Let $\kappa \in (0,1)$, $\alpha \in (-1, \infty)$, $p \in [2,\infty)$, and $\rho \in (0, r_0)$. There exists $\delta_1 = \delta_1(\kappa, \alpha, p, n ,\rho) >0$ sufficiently small such that the following assertions hold. Suppose that $\A$ and $\mu$ satisfy \eqref{mu.def}, \eqref{structure-A}, and
\begin{equation} \label{the-inter}
\Theta_{r, x}(\A, \mu) \leq \delta_1, \quad \forall \ x \in \overline{\Omega^{2\rho}}, \quad \forall r \in (0, R_0)
\end{equation}
for some $R_0 \in (0,1)$, where $\Theta_{r, x}(\A, \mu) $ is defined in \eqref{Theta.def}. Then, for any weak solution $u \in W^1_2(\Omega^\rho, \mu)$ to \eqref{inter-eqn}, if $\F \in L_p(\Omega^\rho, \mu)$, we have $\nabla u \in L_p(\Omega^{2\rho}, \mu)$ and
\[
\|\nabla u\|_{L_p(\Omega^{2\rho}, \mu)} \leq N \|\nabla u\|_{L_2(\Omega^\rho, \mu)} + N\|\F\|_{L_p(\Omega^\rho, \mu)},
\]
where $N >0$ is a constant depending on $\rho,  p, n, \kappa, R_0, \alpha,{\textup{diam}(\Omega^{\rho})}$, and the modulus continuity of $\mu$ on $\overline{\Omega^{r_0}}$.
\end{theorem}

\begin{proof} Let us define
\[
\tilde{\A}(x, \xi) = \mu(x) \A(x, \xi) \quad \text{and} \quad \tilde{\F}(x) = \mu(x) \F(x).
\]
We observe that the vector field $\tilde{\A}$ is uniformly elliptic and bounded in $\Omega^\rho$ with the ellipticity and boundedness constants depend on $\kappa, \rho$, and $\alpha$. For any $x_0 \in \overline{\Omega^{2\rho}}$ and for $r \in (0, \min\{R_0, \rho\})$, we write
\[
\Theta_{r,x_0}(\tilde{\A}) = \frac{1}{|B_r(x_0)|} \int_{B_r(x_0)} \sup_{\xi \in \mathbb{R}^n\setminus\{0\}} \frac{|\tilde{\A}(x,\xi) - \tilde{\A}_{r,x_0}(\xi)|}{|\xi|} \,dx,
\]
where
$$
\tilde{\A}_{r,x_0}(\xi)=\fint_{B_r(x_0)}\tilde{\A}(x,\xi)\,dx.
$$
As in $\Omega^\rho$, $\mu = \mathcal{O}(1)$, we see that $W^1_2(\Omega^\rho, \mu) = W^1_2(\Omega^\rho)$. Then, $u \in W^1_2(\Omega^\rho)$ is a weak solution of
\[
\div[\tilde{\A}(x, \nabla u(x))] = \div[\tilde{\F}(x)] \quad \text{in} \quad \Omega^{\rho}.
\]
Therefore, by the standard $L_p$-regularity theory for uniformly elliptic equations {(see \cite[Theorem 5]{BCGOP} and \cite[Theorem 1.1] {TP} for instance)}, there exists a sufficiently small $\epsilon = \epsilon(\kappa, n, p) >0$ such that if
\begin{equation} \label{Theta-tA}
\Theta_{r, x_0} (\tilde{\A}) \leq \epsilon, \quad \forall \ x_0 \in \overline{\Omega^{2\rho}}, \quad \forall \ r \in (0, \rho_1)
\end{equation}
 for some $\rho_1 \in (0, \min\{R_0, \rho\})$, then
\[
 \int_{B_r(x_0)} |\nabla u(x)|^p \,dx  \leq N r^{n(1-p/2)}\left(\int_{B_{2r}(x_0)} |\nabla u(x)|^2 \,dx\right)^{p/2} +  N\int_{B_{2r}(x_0)}|\tilde{\F}(x)|^p \,dx
\]
for all $x_0 \in \overline{\Omega^{2\rho}}$ and for $r >0$ such that $B_{2r}(x_0) \subset \Omega^{\rho}$, where
for $N = N(p, n, \kappa, \alpha) >0$. From this, and by covering $\overline{\Omega^{2\rho}}$ by a finite number of balls, we obtain
\[
\int_{\Omega^{2\rho}} |\nabla u(x)|^p \,dx \leq N \left(\int_{\Omega^{\rho}} |\nabla u(x)|^2\,dx\right)^{p/2} + N\int_{\Omega^{\rho}}|\tilde{\F}(x)|^p \,dx
\]
for $N = N(p, n, \kappa, \alpha, \text{diam}(\Omega), \rho_1) >0$. From this and as $\mu(x) =\mathcal{O}(1)$, we see that if \eqref{Theta-tA} holds, then
\[
\|\nabla u\|_{L_p(\Omega^{2\rho}, \mu)} \leq N \|\nabla u\|_{L_2(\Omega^\rho, \mu)} + N\|\F\|_{L_p(\Omega^\rho, \mu)}.
\]

\smallskip
It remains to prove that under {the condition} \eqref{the-inter} and with suitable choices of $\delta_1$ and $\rho_1>0$,  \eqref{Theta-tA} holds. For any $\xi \in \mathbb{R}^n\setminus \{0\}$ and $x_0 \in \overline{\Omega^{2\rho}}$, recall the definition of the weighted average $ \A_{B_r(x_0)}(\xi)$ given in \eqref{A-weight-everage} and also let
\[
\mu_{r,x_0} = \frac{1}{|B_r(x_0)|}\int_{B_r(x_0)} \mu(x)\, dx.
\]
Then, we have
\[
\begin{split}
& \frac{|\tilde{\A}(x, \xi) - \mu_{r,x_0} \A_{B_r(x_0)}(\xi) |}{|\xi|}  \\
& \leq  \frac{\mu(x)| \A(x, \xi) - \A_{B_r(x_0)}(\xi)|}{|\xi|} + \frac{|\A_{B_r(x_0)}(\xi)|  |\mu(x) -\mu_{r,x_0} | }{|\xi|} \\
& \leq \frac{\mu(x)| \A(x, \xi) - \A_{B_r(x_0)}(\xi)|}{|\xi|} + N(\kappa)  |\mu(x) -\mu_{r,x_0} |, \quad\forall\, x \in \Omega.
\end{split}
\]
As a result,
\begin{align} \notag
& \fint_{B_r(x_0)} \sup_{\xi\in \mathbb{R}^n\setminus\{0\}}\frac{|\tilde{\A}(x,\xi) - \tilde{\A}_{r,x_0}(\xi)|}{|\xi|} \,dx \\ \notag
&\leq 2\fint_{B_r(x_0)} \sup_{\xi\in \mathbb{R}^n\setminus\{0\}}\frac{|\tilde{\A}(x, \xi) - \mu_{r,x_0} \A_{B_r(x_0)}(\xi) |}{|\xi|} \,dx \\ \notag
& \leq 2 \fint_{B_r(x_0)} \mu(x) \sup_{\xi\in \mathbb{R}^n\setminus\{0\}}\frac{| \A(x, \xi) - \A_{B_r(x_0)}(\xi)|}{|\xi|} \,dx + N(\kappa) \sup_{x, y \in B_r(x_0)} |\mu(x) -\mu(y)| \\ \notag
& \leq N(\alpha, n ,\rho) \fint_{B_r(x_0)} \frac{| \A(x, \xi) - \A_{B_r(x_0)}(\xi)|}{|\xi|} d \mu(x) + N(\kappa) \sup_{x, y \in B_r(x_0)} |\mu(x) -\mu(y)| \\  \label{las.0419}
& \leq N(\alpha, n ,\rho) \Theta_{r,x_0}(\A, \mu) +  N(\kappa) \sup_{x, y \in B_r(x_0)} |\mu(x) -\mu(y)| .
\end{align}
Next, we choose $\delta_1 = \delta_1(\kappa, \alpha, n, p ,\rho)>0$ sufficiently small such that
\[
[N(\alpha, n ,\rho) + N(\kappa)]\delta_1 \leq \epsilon.
\]
As $\mu$ is uniformly continuous on $\overline{\Omega^{\rho}}$, we can find $\rho_1 \in (0, \min\{R_0, \rho\})$  sufficiently small such that
\[|\mu(x) - \mu(y)| \leq \delta_1, \quad \forall x, y \in \overline{\Omega^{\rho}},\,\, |x -y| \leq \rho_1. \]
Then with this choice of $\rho_1$, \eqref{Theta-tA} follows from \eqref{the-inter} and \eqref{las.0419}. The proof of the theorem is completed.
\end{proof}
%====
\subsection{Local boundary estimates}
Recall that for $R>0$, $B_R'(x')$ denotes the ball in $\mathbb{R}^{n-1}$ of radius $R$ centered at $x' \in \mathbb{R}^{n-1}$, and also $B_R' = B_R'(0)$.  Let $r_0 \in (0,1)$ be as in \eqref{mu.def}. As $\Omega$ is $(\delta, \rho_0)$-Lipschitz,  it follows that for $R \in (0, \min\{\rho_0, r_0\}/2]$, and $x_0 \in \partial \Omega$, by Definition \ref{Lip-domain} and with a rotation and translation, we may assume that $x_0=0$ and
\[
\begin{split}
& \mathcal{C}_{2R} :=\big\{ x = (x', x_n) \in B_{2R}' \times \mathbb{R}:  \gamma(x') < x_n < \gamma(x') + 2R\big \} \subset \Omega \quad \text{and} \\
& \Gamma_{2R}:= \Big \{ (x', \gamma(x')): x' \in B_{2R}' \Big\} \subset \partial \Omega,
\end{split}
\]
where $\gamma: \overline{B_{2R}'} \rightarrow \mathbb{R}$ is a Lipschitz function which satisfies
\begin{equation*}
\gamma(0) = 0, \quad \nabla \gamma (0) =0, \quad \text{and} \quad \|\nabla \gamma\|_{L_\infty(B_{2R}')} \leq \delta.
\end{equation*}
Recall also that
\[
\mu(x) = \text{dist}(x, \partial \Omega)^\alpha, \quad x \in \mathcal{C}_{2R}.
\]
In this subsection, we study the equation \eqref{main-eqn} locally near $0 \in \partial \Omega$:
\begin{equation} \label{bdr.eqn}
\left\{
\begin{array}{cccl}
\div[\mu(x) \A(x, \nabla u(x))] & = & \div(\mu(x) \F ) & \quad \text{in} \quad \mathcal{C}_{2R}, \\
\mu(x) \A(x, \nabla u(x)) \cdot \vec{\nu}  & = & 0 & \quad \text{on} \quad \Gamma_{2R}.
\end{array} \right.
\end{equation}
For $p \in (1, \infty)$, we say  $u \in W^1_p(\mathcal{C}_{2R}, \mu)$ is a weak solution of \eqref{bdr.eqn} if
\begin{equation} \label{local-weak-form}
\int_{\mathcal{C}_{2R}} \mu(x) \wei{\A(x,\nabla u(x)), \nabla \varphi(x)} \,dx = \int_{\mathcal{C}_{2R}} \mu(x) \wei{\F(x), \nabla \varphi(x)} \,dx
\end{equation}
for all $\varphi \in C^\infty(\overline{\mathcal{C}}_{2R})$ vanishing on the neighborhood of $\partial \mathcal{C}_{2R} \setminus \Gamma_{2R}$.

This subsection is devoted to the proof of the following result on local regularity estimates of weak solutions to \eqref{bdr.eqn}.

\begin{theorem} \label{local-Lp-thm} Let $R \in (0, \min\{\rho_0, r_0\}/2)$. For each $p \in [2, \infty)$, there exists a sufficiently small constant $\delta_2 =\delta_2(\kappa, p, n, \alpha)>0$ such that if \eqref{structure-A} holds, $\Omega$ is $(\delta, \rho_0)$-Lipschitz with $\delta \in (0, \delta_2)$, and
\begin{equation} \label{ep1}  \Theta_{\rho, x_0}(\A, \mu) < \delta_2 \quad \forall x_0 \in \overline{\mathcal{C}_R}, \ \forall \ \rho \in (0, R_0)
\end{equation}
for some $R_0\in(0,1)$, then for any weak solution $u \in W^1_2(\mathcal{C}_{2R}, \mu)$ of \eqref{bdr.eqn} with $\F \in L_p(\mathcal{C}_{2R}, \mu)^n$, we have $\nabla u \in L_p(\mathcal{C}_R, \mu)$ and
\begin{equation} \label{Du-est-brx}
\begin{split}
&\left(\int_{\mathcal{C}_{R}} |\nabla u(x)|^p \mu(x) \,dx\right)^{1/p} \\
&  \leq N \mu(\mathcal{C}_{2R})^{\frac{1}{p} -\frac{1}{2}}\left(\int_{\mathcal{C}_{2R}} |\nabla u(x)|^2 \mu(x) \,dx\right)^{1/2}+ N\left(\int_{\mathcal{C}_{2R}} |\F(x)|^p \mu(x) \,dx\right)^{1/p},
\end{split}
\end{equation}
where $N = N(\kappa, \alpha, n, p, R_0) >0$.
\end{theorem}
To prove Theorem \ref{local-Lp-thm}, we flatten the boundary $\Gamma_{2R}$ and then apply Theorem \ref{flat-Lp-thrm}. We begin with the following simple lemma on the properties of the weight $\mu$.

\begin{lemma}  \label{mu-lemma}  Assume that $\Omega$ is $(\delta, \rho_0)$-Lipschitz.  Then, there exist $h: C_{2R} \rightarrow \mathbb{R}$ and $N = N(\alpha)>0$ satisfying
\begin{equation*} %\label{mu.eqn}
\mu(x) = \big(x_n -\gamma(x')\big)^\alpha \big(1- h(x)\big)^\alpha \quad\text{and} \quad 0 \leq h(x) \leq \delta
\end{equation*}
for all $x \in \mathcal{C}_{2R}$.
\end{lemma}
\begin{proof}
As $R \in (0, \min\{r_0, \rho_0\}/2)$, for each $x = (x', x_n) \in \mathcal{C}_{2R}$, we have
\[
\mu(x) = d^{\alpha}(x) \quad \text{and} \quad d(x) = \inf_{\xi \in \overline{B_{2R}'}} \big( |x' -\xi|^2 + |x_n - \gamma(\xi)|^2\big)^{1/2} .
\]
By the definition, it is clear that
\[
d(x) \leq x_n - \gamma(x') \quad \forall \ x = (x', x_n) \in \mathcal{C}_{2R}.
\]
On the other hand, as the cone with vertex at $(\gamma(x'), x_n)$ and slope $\delta$ stays above the graph of $\Gamma_{2R}$, we also have
\[
d(x) \geq \frac{x_n -\gamma(x')}{\sqrt{1+\delta^2}}.
\]
From the last two estimates and by taking
\[ h(x) = 1- \frac{d(x)}{x_n -\gamma(x')}, \quad x= (x', x_n) \in \mathcal{C}_{2R},
\]
we see that
\[
0 \leq h(x) \leq 1 - \frac{1}{\sqrt{1+\delta^2}} \leq  \delta.
\]
The lemma is proved.
\end{proof}
%==========
Next, we flatten the boundary $\Gamma_{2R}$ and transfer the equation \eqref{bdr.eqn} into the equation in the upper-half space as in \eqref{D2.eqn}. Let $\Phi: \mathcal C_{2R} \rightarrow D_{2R}^+: = B_{2R}' \times (-2R, 2R)$ and $\Psi : D_{2R}^+ \rightarrow \mathcal C_{2R}$ be defined by
\[
\begin{split}
& \Phi(x) = (x', x_n -\gamma(x')) \quad \forall \ x = (x', x_n) \in \mathcal C_{2R}, \\
& \Psi(y) = (y', y_n + \gamma(y'))  \quad \forall y = (y', y_n) \in D_{2R}.
\end{split}
\]
By a simple calculation, we see that
\begin{equation} \label{na-phi}
\nabla \Phi (x) =
\left(
\begin{matrix}
\I_{n-1} & 0 \\
-\partial_{x'}\gamma(x') & 1
\end{matrix}
\right)
\end{equation}
for all $x = (x', x_n) \in \mathcal C_{2R}$, and
\begin{equation} \label{na-Psi}
\nabla \Psi(y) = \left(
\begin{matrix}
\I_{n-1} & 0 \\
\partial_{y'}\gamma(y')  & 1
\end{matrix}
\right)
\end{equation}
for $y = (y', y_n) \in D_{2R}$, where $\I_{n-1}$ is the $(n-1)\times (n-1)$ identity matrix.
We note that $\text{det} (\nabla \Psi) = \text{det}(\nabla \Phi) =1$ and
\[
\Phi = \Psi^{-1}, \quad \nabla \Psi (y) = [\nabla \Phi(\Psi(y)]^{-1}, \quad \forall \ y \in D_{2R}^+.
\]
Moreover, as $\delta \in (0,1)$,
\begin{equation} \label{grad-flaten}
\begin{split}
& \|\nabla \Phi\|_{L_\infty(\mathcal C_{2R})}^2 \leq n + \|\nabla \gamma\|_{L_\infty(B_R')}^2 \leq n + \delta^2 \leq n+1\quad \text{and} \\
& \|\nabla \Psi\|_{L_\infty(D_{2R}^+)}^2 \leq n+1.
\end{split}
\end{equation}
Now, let us recall
\[ T_{2R} = B_{2R}' \times \{0\} = \partial D_{2R}^+ \cap \{y_n =0\},
\]
and denote
\begin{equation} \label{w-hatA.def}
\omega(y) = \mu (\Psi(y)) \quad \text{and} \quad  \hat{\A}(y, \xi) =  \A(\Psi(y), \xi [\nabla \Phi(\Psi(y))]) [\nabla \Phi(\Psi(y))]^*
\end{equation}
for $y \in D_{2R}^+$. We then consider the equation
\begin{equation} \label{w.eqn0406}
\left\{
\begin{array}{cccl}
\textup{div}[\omega(y) \hat{\A}(y, \nabla w(y))] & = & \textup{div}[\omega(y)\G(y)] &\quad \text{in} \quad D_{2R}^+,\\
\omega(y) \hat{A}_n (y, \nabla w(y)) & = & \omega(y)G_n(y) & \quad \text{on} \quad T_{2R}.
\end{array} \right.
\end{equation}
We note that a function $w \in W^{1}_{p}(D_{2R}^+, \omega)$ is said to be a weak solution of \eqref{w.eqn0406} if
\begin{equation} \label{wea-w}
\int_{D_{2R}^+} \omega(y) \wei{\hat{\A}(y, \nabla w(y)), \nabla \varphi(y)} dy = \int_{D_{2R}^+} \omega(y) \wei{\G(y), \nabla \varphi(y)} dy
\end{equation}
for all $\varphi \in C^\infty({\overline{D_{2R}^+}})$ which vanishes on the neighborhood of $\partial D_{2R}^+ \setminus T_{2R}$.
%=====
\begin{lemma} \label{change-cor}
Assume that $\Omega$ is $(\delta, \rho_0)$-Lipschitz. If $u \in W^1_p(\mathcal{C}_{2R}, \mu)$ is a weak solution of \eqref{bdr.eqn} for some $p \in (1, \infty)$, then for
\[ w(y)  = u(\Psi(y)), \quad y  \in D_{2R}^+,
\]
we have $w  \in W^1_p(D_{2D}^+, \omega)$ is weak a solution of \eqref{w.eqn0406} with
\[
 \G(y) = \F(\Psi(y)) [\nabla\Phi(\Psi(y))]^*, \quad y \in D_{2R}^+.
\]
\end{lemma}
\begin{proof} Since $u \in W^{1}_{p}(\mathcal{C}_{2R}, \mu)$ and by a change of variables, we see that $w \in W^1_p(D_{2R}^+, \omega)$. The lemma follows directly by writing the solutions in the weak forms \eqref{local-weak-form} and \eqref{wea-w}, and using  a change of variables.
\end{proof}
Next, let us denote
\[
[\A]_{\textup{BMO}_{R_0}(\mathcal{C}_R, \mu)} = \sup_{\rho \in (0, R_0)} \sup_{x\in \mathcal{C}_R} \Theta_{\rho,x}(\A, \mu),
\]
where $\Theta_{\rho,x}(\A, \mu)$ is defined in \eqref{Theta.def}. A similar definition can be made also for $[\hat{\A}]_{\textup{BMO}_{R_0}(D_R^+, \omega)}$. Our next result gives the estimate of the mean oscillation of $\hat{\A}$ with the weight $\mu$.
\begin{lemma} \label{A-h-BMO}
Assume that $\Omega$ is $(\delta, \rho_0)$-Lipschitz. If $\A$ satisfies \eqref{structure-A}, then so does $\hat{\A}$ on $D_{2R}^+$. Moreover, there is $N_0 = N_0(\kappa, n) >0$ such that
\[
[\hat{\A}]_{\textup{BMO}_{R_0}(D_R^+, \omega)} \leq N_0 \Big([\A]_{\textup{BMO}_{2R_0}(\mathcal{C}_R, \mu)} + \delta\Big),
\]
where $\hat{\A}$ is defined in \eqref{w-hatA.def}.
\end{lemma}
\begin{proof} The first assertion of the lemma follows directly from a direct calculation, so we skip it. To prove the second assertion in the lemma, we observe that by the mean value theorem, there is $\eta \in \mathbb{R}^n$ such that
\[
\A(\Psi(y), \xi[\nabla \Phi(\Psi(y))]) = \A(\Psi(y), \xi) + \A_{\xi}(\Psi(y), \eta) [\nabla \Phi(\Psi(y)) -\I_n]\xi,
\]
where $\I_n$ is the $n\times n$ identity matrix.
Then, we can write
\[
\hat{\A}(y, \xi) = \mathbb{B}(y,\xi) +  \mathbb{D}(y, \xi),
\]
where
\[
\mathbb{B} = \A(\Psi(y), \xi [\nabla \Phi(\Psi(y))])[\nabla \Phi(\Psi(y)) -\mathbb{I}_n]^* +  \A_{\xi}(\Psi(y), \eta) [\nabla \Phi(\Psi(y)) -\I_n]\xi
\]
and $\mathbb{D}(y, \xi) = \A(\Psi(y), \xi)$. Then, it follows from the boundedness and the growth condition of $\A$ in \eqref{structure-A}, and the explicit formulas in \eqref{na-phi} and \eqref{na-Psi}  that
\[
[\mathbb{B}]_{\textup{BMO}_{R_0}(D_R^+,\omega)} \leq N(n,\kappa) \|\nabla \gamma\|_{L_\infty} \leq N(n, \kappa) \delta.
\]
By a change of variables and subtracting the weighted average, we also have
\[
[\mathbb{D}]_{\textup{BMO}_{R_0}(D_R^+,\omega)} \leq N(n) [\A]_{\textup{BMO}_{2R_0}(C_R,\mu)}.
\]
The proof of the lemma is then completed.
\end{proof}

Now we give the proof of Theorem \ref{local-Lp-thm}.

\begin{proof}[Proof of Theorem \ref{local-Lp-thm}] Let $\delta_0 = \delta_0(\kappa, \alpha, n, p) >0$ be the number defined in Theorem \ref{flat-Lp-thrm}. Choose $\delta_2\in (0, \delta_0)$ such that $2N_0\delta_2 < \delta_0$, where $N_0$ is a number defined in Lemma \ref{A-h-BMO}. Then, it follows from $\delta\leq \delta_2$, \eqref{ep1}, and Lemma \ref{A-h-BMO} that
\[
\sup_{\rho \in (0, R_0/2)}\sup_{x \in D_R^+} \Theta_{\rho, x}(\hat{\A}, \omega) \leq \delta_0.
\]
From this and Definition \ref{PBMO-def}, we apply Theorem \ref{flat-Lp-thrm} to the equation \eqref{w.eqn0406} with a scaling and obtain
\[
\begin{split}
& \left(\fint_{D_R^+} |\nabla w(y)|^p \,d\omega(y) \right)^{1/p} \\
& \leq N \left(\fint_{D_{2R}^+} |\nabla w(y)|^2 \,d\omega(y) \right)^{1/2} + N \left(\fint_{D_{2R}^+} | \mathbf{G}(y)|^p \,d\omega(y) \right)^{1/p}
\end{split}
\]
for $N = N(\kappa, \alpha, p,n)>0$. From this, the definition of $w$ and $\mathbf{G}$ in Lemma \ref{change-cor}, and the estimates in \eqref{grad-flaten},  \eqref{Du-est-brx} follows by using the change of variables $y \mapsto x = \Psi(y)$. The proof of Theorem \ref{local-Lp-thm} is completed.
\end{proof}

\subsection{Global \texorpdfstring{$L_p$}{Lp}-estimates}  This subsection gives the proof of Theorem \ref{g-theorem}.
%====

\begin{proof}[Proof of Theorem \ref{g-theorem}]  Let $\rho = \min\{\rho_0, r_0\}/8$ and $\delta = \min\{\delta_1, \delta_2 \}$,  where $\delta_1 = \delta_1(\kappa, p, n ,\alpha,\rho)$ is defined in Theorem \ref{inter.est-thm} and $\delta_2 = \delta_2(\kappa, p, n, \alpha)$ is defined  Theorem \ref{local-Lp-thm}. We prove Theorem \ref{g-theorem} with this choice of $\delta$.

\smallskip
Note that as $p \geq 2$,  $\F \in L_2(\Omega, \mu)$ if $\F \in L_p(\Omega, \mu)$. Then, by the Minty-Browder theorem, it follows that there exists a weak solution $u \in W^1_2(\Omega, \mu)$ to the equation \eqref{main-eqn}, and this weak solution is unique up to a constant. More precisely, let $\mathcal{X}$ be the space consisting of all functions $v \in L_{1,\text{loc}}(\Omega)$ such its weak derivative $\nabla v$ exists,
\[
\left(\int_\Omega |\nabla v(x)|^2\,d\mu(x)\right)^{1/2}<\infty, \quad \text{and} \quad \int_{B} v(x) dx =0
\]
for some fixed ball $B \subset \Omega$. The space $\mathcal{X}$ is endowed with the norm
\[
\|v\|_{\mathcal{X}}=\left(\int_\Omega |\nabla v(x)|^2\,d\mu(x)\right)^{1/2}, \quad v \in \mathcal{X}.
\]
It is easy to show that $\mathcal{X}$ is a Banach space and the norm is uniformly convex. Therefore, $\mathcal{X}$ is reflexive and the Minty-Browder theorem is applicable which gives the existence of a weak solution $u \in W^1_2(\Omega, \mu)$ to the equation \eqref{main-eqn}.

\smallskip
It remains to prove \eqref{main-est}. Observe that with our choice of $\delta$ and under the assumptions of Theorem \ref{g-theorem},  the conditions in Theorems \ref{inter.est-thm} and  \ref{local-Lp-thm} are satisfied. Due to this, we apply Theorem \ref{inter.est-thm} to get
\[
\|\nabla u\|_{L_p(\Omega^{2\rho}, \mu)} \leq N \|\nabla u\|_{L_2(\Omega, \mu)} + N\|\F\|_{L_p(\Omega, \mu)}.
\]
Similarly, applying Theorem \ref{local-Lp-thm}, we obtain
\[
\|\nabla u\|_{L_p(\mathcal{C}_{3\rho}(x_0), \mu)} \leq N \|\nabla u\|_{L_2(\Omega, \mu)} + N\|\F\|_{L_p(\Omega, \mu)}
\]
for any $x_0 \in \partial \Omega$. Then it follows from the compactness of $\overline{\Omega}$ that
\[
\|\nabla u\|_{L_p(\Omega, \mu)} \leq N\|\nabla u\|_{L_2(\Omega, \mu)} + N\|\F\|_{L_p(\Omega, \mu)}
\]
for $N>0$ depending on $p,\kappa, n, r_0, R_0, \rho_0, \alpha, \Omega$,  and the modulus of continuity of $\mu$ on $\overline{\Omega^\rho}$. On the other hand, by  the energy estimate and H\"{o}lder's inequality, we have
\[
\|\nabla u\|_{L_2(\Omega, \mu)} \leq N(\kappa)\|\F\|_{L_2(\Omega, \mu)} \leq N(\kappa, p, \Omega)\|\F\|_{L_p(\Omega, \mu)}.
\]
Therefore,
\[
\|\nabla u\|_{L_p(\Omega, \mu)} \leq N  \|\F\|_{L_p(\Omega, \mu)}
\]
and the first assertion in \eqref{main-est} is proved.

\smallskip
It remains to prove the second assertion in \eqref{main-est}. To this end, we apply the weighted Sobolev embedding theorem \cite[Remark 3.2 (ii)]{DP20}, see also \cite[Theorem 6]{hajlasz}.  In fact, by flattening the  boundary of the domain $\Omega$, and using Lemma \ref{mu-lemma} and  a partition of unity, we can apply the Sobolev embedding \cite[Remark 3.2 (ii)]{DP20} to obtain
\begin{equation} \label{Sob-embed}
\|u\|_{L_{p_1}(\Omega, \mu)}  \leq N \|u\|_{W^1_2(\Omega, \mu)},
\end{equation}
where $p_1 \in (2, \infty]$ satisfying
\[
\frac{n+\alpha_+}{2} \leq 1 + \frac{n + \alpha_+}{p_1}
\]
and $N = N(\Omega, \alpha, r_0)>0$. Then, by the energy estimate and H\"{o}lder's inequality, we infer from \eqref{Sob-embed} that
\[
\begin{split}
\|u\|_{L_{p_1}(\Omega, \mu)} & \leq N \|u\|_{L_2(\Omega, \mu)}  + N\|\F\|_{L_2(\Omega, \mu)} \\
& \leq N \|u\|_{L_2(\Omega, \mu)}  +N \|\F\|_{L_p(\Omega, \mu)}.
\end{split}
\]
If $p_1 \geq p$, the second estimate in \eqref{main-est} follows. Otherwise, we repeat the process by applying the Sobolev embedding \cite[Remark 3.2 (ii)]{DP20} again to obtain
\[
\|u\|_{L_{p_2}(\Omega, \mu)} \leq N \|u\|_{W^{1}_{p_1}(\Omega, \mu)} \leq N \|u\|_{L_2(\Omega, \mu)}  + N\|\F\|_{L_p(\Omega, \mu)},
\]
with $p_2 \in (p_1, \infty]$ satisfying
\[
\frac{n+\alpha_+}{p_1} \leq 1 + \frac{n + \alpha_+}{p_2}.
\]
By doing this, we obtain an increasing sequence of numbers $\{p_k\}_k$ defined as above and obtain the second estimate in  \eqref{main-est} when $p_k \geq p$ for some $k \geq 1$. The proof of the theorem is completed.
\end{proof}
%We now conclude the paper by the
\begin{remark} By using the Sobolev embedding, H\"older's inequality, and a standard iteration argument, the second estimate in \eqref{main-est} can be replaced by
\[
\|u\|_{L_{p^*}(\Omega, \mu)} \leq N \Big[\|u\|_{L_{1}(\Omega, \mu)} + \|\F\|_{L_p(\Omega, \mu)} \Big]
\]
for $p^* \in (p, \infty)$ satisfying
\[
\frac{n+\alpha_+}{p} \leq 1 + \frac{n + \alpha_+}{p^*}
\]
and if the strict inequality holds $p^*=+\infty$ is allowed.
We also point out that the weighted Poincar\'e inequality of the type
\[
\|u - \bar{u}_{\Omega}\|_{L_{p}(\Omega, \mu)} \leq N \|\nabla u\|_{L_p(\Omega,\mu)}, \quad \text{where} \quad \bar{u}_{\Omega} = \frac{1}{\mu(\Omega)}\int_{\Omega} u(x)\, d\mu(x)
\]
obtained in \cite{Fabes} cannot be directly applied as $\mu \not\in A_p$ when $\alpha \geq p-1$.
\end{remark}
%====

\end{document}